\documentclass[11pt,a4paper]{amsart}
\usepackage[T1]{fontenc}
\usepackage[utf8]{inputenc}
\usepackage{lmodern,amsmath,amssymb,amsthm,mathtools,booktabs,array,longtable}
\usepackage[margin=29mm]{geometry}
\usepackage{xcolor}
\usepackage[colorlinks=true,linkcolor=blue,citecolor=blue,urlcolor=blue]{hyperref}
\hypersetup{pdftitle={Okubo algebras and Exceptional Periodicity: tetracode gradings, root-lattice obstructions and Clifford cones},
pdfauthor={A. Marrani and D. Corradetti}}
\newcommand{\C}{\mathbb C}\newcommand{\R}{\mathbb R}\newcommand{\Z}{\mathbb Z}
\newcommand{\F}{\mathbb F}\newcommand{\g}{\mathfrak g}\newcommand{\h}{\mathfrak h}
\newcommand{\so}{\mathfrak{so}}\newcommand{\slie}{\mathfrak{sl}}
\newcommand{\tr}{\operatorname{tr}}\newcommand{\Span}{\operatorname{span}}
\newcommand{\Aut}{\operatorname{Aut}}\newcommand{\Der}{\operatorname{Der}}
\newcommand{\Cl}{\operatorname{Cl}}
\newcommand{\HSpin}{\operatorname{HSpin}}
\newcommand{\eps}{\varepsilon}\newcommand{\mc}{\mathcal}
\DeclareMathOperator{\Hom}{Hom}\DeclareMathOperator{\rank}{rank}
\newtheorem{theorem}{Theorem}[section]
\newtheorem{proposition}[theorem]{Proposition}
\newtheorem{lemma}[theorem]{Lemma}
\newtheorem{corollary}[theorem]{Corollary}
\theoremstyle{definition}
\theoremstyle{remark}\newtheorem{remark}[theorem]{Remark}
\numberwithin{equation}{section}
\title[Okubo algebras and Clifford cones]{Okubo algebras and Exceptional Periodicity:\\
tetracode gradings, root-lattice obstructions and Clifford cones}
\author[A. Marrani]{Alessio Marrani}
\address{Centre for Mathematics and Theoretical Physics, University of Hertfordshire, Hatfield AL10 9AB, UK}
\email{a.marrani@herts.ac.uk}
\author[D. Corradetti]{Daniele Corradetti}
\address{Grupo de F\'isica Matem\'atica, Instituto Superior T\'ecnico, Av. Rovisco Pais, 1049-001 Lisboa, Portugal}
\email{danielecorradetti@tecnico.ulisboa.pt}
\address{Departamento de Matem\'atica, Universidade do Algarve, Campus de Gambelas, 8005-139 Faro, Portugal}
\email{a55499@ualg.pt}
\begin{document}
\begin{abstract}
The Okubo algebra brings together ternary gradings, integral root frames
and rank-three geometry. We examine how these structures extend to the
complex Exceptional Periodicity algebras. With $N=4(n+1)$, a tetracode
character gives eight nonzero homogeneous components of dimension
$t_n=8n+3+2^{4n}$. On a tip, a compatible grading has zero component
of dimension $8n-5$ and eight nonzero components of dimension
$1+2^{4n-3}$. The arithmetic behaves differently: for $n\geq2$,
the lattice $D_N^+$ has no full-rank orthogonal frame of unscaled
$A_2$ lattices. We exhibit failures of the Jacobi and Jordan identities
and determine the derivations and full automorphism group of the
ambient algebra above level one. We also give a Clifford-module proof,
with the normalization fixed by root products, of the published
realization of EP tips as special Vinberg $H$-algebras. This identifies
a positive real form and yields explicit formulas for its cone, cubic,
Cholesky coordinates and rational dual invariant. The cone is symmetric
exactly at $n=1$. Its determinant-one hypersurface belongs to the known
$L(8n,1,0)$ series, whose $r$- and $c$-maps give homogeneous scalar
manifolds for supergravity theories with eight real supercharges.
\end{abstract}
\maketitle
\raggedbottom
\section{Introduction}
Octonions link exceptional Lie algebras, Jordan algebras, lattices and
projective geometry. The para-octonion and Okubo products retain the
composition norm but change the multiplication. We work with these
three eight-dimensional models; the wider class of nonunital
composition isotopes lies outside this comparison.

Three features of the compact Okubo algebra guide the discussion.
Its complexification has a Pauli grading by $\F_3^2$, with eight
one-dimensional nonzero components. In the integral Petersson model
of \cite{corradettiIntegral}, twelve nonzero idempotents generate an
$A_2^4$ sublattice of $E_8$, and the tetracode supplies its glue.
The classical Magic Star, meanwhile, has six 27-dimensional tips,
each carrying an Albert product once a basepoint is chosen. These
features involve different choices of field, multiplication and
integral data. We keep those choices explicit throughout.

Exceptional Periodicity (EP) extends the root-coordinate description
of $E_8$ to $N=4(n+1)$ coordinates \cite{ep1,ep2,epPublished,epStar,
epQuantum,marraniT}. Above level one the resulting anticommutative
algebra fails the Jacobi identity. We study its character gradings,
root lattice and tip products. The tetracode gives a balanced grading
at every level, whereas an unscaled full-rank $A_2$ frame occurs only
at level one. Even there, the character kernel and the idempotent
lattice are different: their root systems are $D_4$ and $A_2^4$,
respectively.

For the geometry we start from the published realization of the tip
as a special Vinberg $H$-algebra \cite[Section 8, Remark 8.2 and
Appendix A]{epPublished}. Definition 8.4 fixes the product, and
Propositions 8.8 and A.2 give its unit and associative trace form.
In Section~\ref{sec:cone} we match this realization to our frame and
derive the Clifford relations directly from the root product. A
single root calculation fixes their scalar coefficient. The resulting
isomorphism transports the positive real Clifford model to the tip,
so the cone theory of \cite{ac,ams,amsErr,cubics,ac2025} applies.
This real form is a separate choice from the complex character grading.

The cubics also have a familiar physical interpretation. A linear
change of coordinates places them in the $L(8n,1,0)$ subsequence of
the homogeneous special real spaces classified by de Wit and
Van Proeyen \cite{dwp}. They define five-dimensional ungauged
Maxwell--Einstein supergravity models; the $r$- and spacelike $c$-maps
then give four- and three-dimensional scalar geometries. The input
to these models is the cubic on its positive real domain. Their
gauge and duality symmetries are ordinary Lie symmetries.

\begin{theorem}[Principal statements]\label{th:main}
Let $n\geq1$, $N=4(n+1)$, and let $\g_n$ be the complex EP algebra
defined in Section~\ref{sec:roots}.
\begin{enumerate}
\item A tetracode character gives
$\dim(\g_n)_c=t_n=8n+3+2^{4n}$ for every $0\ne c\in\F_3^2$.
A compatible tip product has homogeneous dimensions
$8n-5$ at zero and $1+2^{4n-3}$ at every nonzero degree.
\item Its root lattice is the unimodular lattice $D_N^+$, which is even
exactly when $n$ is odd. For $n\geq2$ it contains no full-rank orthogonal
frame of $A_2$ lattices with roots of squared length two.
\item $\g_n$ is Lie exactly at $n=1$. For $n\geq2$ the span of
Jacobiator values is its entire spinorial summand, and
$\Der(\g_n)=\so(2N,\C)$,
$\Aut(\g_n)=\HSpin(2N,\C)$.
\item With the product of Section~\ref{sec:tip}, a complex
tip is Jordan exactly at $n=1$; an explicit Jordanator is nonzero
at every $n\geq2$.
\item There is a frame- and trace-preserving complex algebra
isomorphism $T_n\simeq H_n^{\rm Cl}\otimes_\R\C$.
A specified transported real form carries a homogeneous convex
cone $\Omega_n$ isomorphic to $\Omega_n^{\rm Cl}$,
symmetric exactly at $n=1$. Its cubic yields the known
$L(8n,1,0)$ special real model.
\end{enumerate}
\end{theorem}

Within this family, equal root lengths, the Lie property, an unscaled
full-rank $A_2$ frame, the Jordan property of the tip and symmetry of
its positive cone all occur at $n=1$. Each condition has its own proof;
Theorem~\ref{th:levelone} collects them. The cone uses the real form
specified in Theorem~\ref{th:identification}. The balanced complex
grading, by contrast, survives at every level.

\section{Okubo models, ground fields and the tetracode}
\subsection{Symmetric composition and the Pauli model}
A symmetric composition algebra is an algebra $(A,*,q)$ with a
nondegenerate quadratic form satisfying $q(x*y)=q(x)q(y)$ and the
associativity of its polar form. In characteristic different from two,
these conditions imply
\begin{equation}\label{eq:symmetric}
 (x*y)*x=x*(y*x)=q(x)y.
\end{equation}
For an octonion algebra with standard conjugation and an order-three
automorphism $\tau$, the Petersson product is
\begin{equation}\label{eq:petersson}
 x*y=\tau(\bar x)\,\tau^2(\bar y).
\end{equation}
Suitable $\tau$ give Okubo algebras; the identity automorphism gives the
para-octonions. Classification and gradings are discussed in
\cite{okubo,petersson,elduqueComposition,elduqueOrder,elduqueSymmetric,
elduqueOkubo,elduqueMyung}.

Over $\C$, a useful normalization on traceless $3\times3$ matrices is
\begin{equation}\label{eq:okuboMatrix}
 x*y=\mu xy+\bar\mu yx-\tfrac13\tr(xy)I_3,
 \qquad \mu=\tfrac12+\tfrac{i}{2\sqrt3},\quad
 \mu+\bar\mu=1,\quad \mu\bar\mu=\tfrac13.
\end{equation}
Here the bar on $\mu$ denotes its displayed complex conjugate; the
quadratic form on the complex vector space is
\begin{equation}\label{eq:okuboNorm}
 q(x)=\tfrac16\tr(x^2).
\end{equation}
For example, if $\tr x=0$, the Cayley--Hamilton identity gives
$\tr(x^4)=\frac12\tr(x^2)^2$, and therefore
$q(x*x)=\frac1{36}\tr(x^2)^2=q(x)^2$.
Traceless Hermitian matrices give the compact real model:
\eqref{eq:okuboMatrix} preserves this subspace, and
\eqref{eq:okuboNorm} is positive definite on it. On the full complex
matrix space the same quadratic form has isotropic vectors.

Let $P=\operatorname{diag}(1,\omega,\omega^2)$ and let $Q$ be the
cyclic permutation matrix, where $\omega^3=1\ne\omega$.
The lines $\C P^aQ^b$, $(a,b)\ne(0,0)$, give the Pauli grading
\begin{equation}\label{eq:pauli}
 A_{\C}=\bigoplus_{0\ne c\in\F_3^2}(A_{\C})_c,
 \qquad \dim(A_{\C})_c=1,
 \qquad (A_{\C})_c*(A_{\C})_d\subseteq(A_{\C})_{c+d}.
\end{equation}
The omitted degree-zero component is zero. The additive group
$\F_3^2$ is precisely $\Z_3\times\Z_3$.

\begin{lemma}[A positivity obstruction]\label{lem:realGrading}
A positive definite real symmetric composition algebra admits no
nontrivial grading by $\F_3^2$.
\end{lemma}
\begin{proof}
If a nonzero homogeneous $x$ has degree $c\ne0$, then
$(x*x)*x=q(x)x$. The left-hand side has degree $3c=0$ and the
right-hand side has degree $c$, so $q(x)x=0$. Positive definiteness
contradicts this. Every nonzero component consequently has degree zero.
\end{proof}
The Pauli grading therefore belongs to the complex algebra. A similar
obstruction holds for Euclidean Jordan algebras graded by a finite
group of odd order. If $x$ has nonidentity degree $c$, then
$L_{x^2}$ shifts degrees by $c^2\ne1$, so its matrix has no diagonal
homogeneous block and $\operatorname{Tr}L_{x^2}=0$. In a Euclidean
Jordan algebra this trace is positive for $x\ne0$, a contradiction.
In particular, the complex Albert grading cannot restrict to the
Euclidean real form as the same grading.

\subsection{Integral idempotents and the classical Magic Star}
Fix the integral octonion order and the order-preserving Petersson
automorphism used in \cite{coxeter,corradettiIntegral,corradettiMay}.
The compact octonion, para-octonion and Okubo products then share
the same additive lattice. Its norm-one shell contains 240 elements.
To compare it with the usual $E_8$ root metric, we use squared
root length $2q(x)$, so norm one means squared length two.
The nonzero idempotent counts in this specified model are respectively
$1$, $57$ and $12$. The count excludes the zero idempotent.
Composition implies that a nonzero idempotent has $q(e)=1$;
the finite shell calculation is therefore legitimate.

The twelve Okubo idempotents partition into four orthogonal oriented
triangles, each with three vectors summing to zero. Their span is
\begin{equation}\label{eq:a2frame}
 L_I\simeq A_2^4\subset E_8,\qquad [E_8:L_I]=9.
\end{equation}
The order, multiplication and idempotents are given in
\cite{corradettiIntegral}. The triangle singled out by $\tau$ is
distinguished relative to this Petersson datum. An abstract Okubo
algebra carries neither that chosen automorphism nor that integral order. An external $A_2$ subsystem in $E_8$ gives the classical
root count and complex branching rule
\begin{align}
 240&=6+72+6\cdot27,\label{eq:root240}\\
 \mathfrak e_8&=(\mathfrak e_6,\mathbf1)\oplus
 (\mathbf1,\slie_3)\oplus(\mathbf{27},\mathbf3)
 \oplus(\overline{\mathbf{27}},\overline{\mathbf3}).\label{eq:branching}
\end{align}
The bars in the last summand denote dual complex representations.
The six tips are members of Jordan pairs; a unit or basepoint is needed
to turn a member into a unital Albert algebra \cite{loos,truini,marraniTruini}.
The positive real Albert cone has the Cayley plane as the projectivization
of its nonzero rank-one boundary rays. Its open interior consists of
rank-three positive elements and contains no rank-one element.

For either the para-octonion or Okubo symmetric composition model,
the compact triality construction gives the vector-space decomposition
\begin{equation}\label{eq:f4}
 \mathfrak f_{4(-52)}\simeq\so(8)\oplus\mathbf8_v
 \oplus\mathbf8_s\oplus\mathbf8_c.
\end{equation}
The symmetric composition construction defines the bracket on this
space \cite{barton,elduqueComposition,springer}. Triality permutes
the three irreducible eight-dimensional $\so(8)$ representations.

\subsection{The ternary code}
Set
\begin{equation}\label{eq:tetracode}
 H=\begin{pmatrix}1&0&1&1\\0&1&1&-1\end{pmatrix},
 \qquad \mc T=\ker(H:\F_3^4\longrightarrow\F_3^2).
\end{equation}
Since $HH^{\mathsf T}=0$ and $\rank H=2$, the row space of $H$
equals $\mc T=\mc T^\perp$. Each of its eight nonzero words
has weight three. Thus $\mc T$ is the self-dual $[4,2,3]_3$
tetracode with weight enumerator $1+8z^3$ \cite{conway}.
The discriminant group of $A_2^4$ is $\F_3^4$; its isotropic
two-dimensional glue code gives the index-nine even unimodular
extension $E_8$. This discriminant-form construction and the
character kernel in Section~\ref{sec:grading} use isomorphic codes,
but they are different lattice constructions.

\section{The complex root algebra and its lattice}\label{sec:roots}
\subsection{Generalized roots and normalization}
Let $k_1,\ldots,k_N$ be an orthonormal basis and put $N=4(n+1)$.
\begin{equation}\label{eq:roots}
 \begin{split}
 \Phi_O&=\{\pm k_i\pm k_j:1\leq i<j\leq N\},\\
 \Phi_S&=\{\tfrac12\textstyle\sum_{i=1}^N\epsilon_i k_i:
 \epsilon_i=\pm1,\ \#\{i:\epsilon_i=+1\}\text{ even}\},
 \qquad \Phi=\Phi_O\sqcup\Phi_S.
 \end{split}
\end{equation}
Here $\Phi_O$ is the $D_N$ root system and $\Phi_S$ is the weight
set of one half-spin representation of $\so(2N,\C)$. Their union is
the $E_8$ root system at $n=1$; at higher levels we use the term
\emph{generalized roots}.
For exact calculations a root is stored as $a=2\alpha$, and
\begin{equation}\label{eq:doubled}
 (\alpha,\beta)=\tfrac14\sum_i a_i b_i.
\end{equation}
\begin{lemma}\label{lem:length}
Orthogonal roots have squared length $2$ and spinorial roots have
squared length $n+1$. For spinorial roots differing in $d$ signs,
\begin{equation}\label{eq:spinDot}
 (\alpha,\beta)=n+1-\tfrac d2,
 \qquad \alpha+\beta\in\Phi\Longleftrightarrow(\alpha,\beta)=-n,
 \qquad \alpha-\beta\in\Phi\Longleftrightarrow(\alpha,\beta)=n.
\end{equation}
\end{lemma}
\begin{proof}
The first two statements follow from the coordinates. The sum of two
spinorial roots has integral coordinates, all in $\{0,\pm1\}$;
it belongs to $\Phi$ exactly when two coordinates are nonzero.
This means $d=N-2$, and substitution gives $-n$. The difference
case follows by replacing $\beta$ with $-\beta$. For distinct
nonopposite spinorial roots, $d$ ranges over $2,4,\ldots,N-2$,
so the scalar products range over all integers $[-n,n]$.
\end{proof}
The sectors have cardinalities
\begin{equation}\label{eq:dimensions}
 |\Phi_O|=2N(N-1),\quad |\Phi_S|=2^{N-1},\quad
 \dim\g_n=N(2N-1)+2^{N-1}.
\end{equation}

\subsection{An explicit asymmetry function and bracket}
Use the ordered lattice basis
\begin{equation}\label{eq:simple}
 \alpha_i=k_i-k_{i+1}\ (1\leq i\leq N-2),\quad
 \alpha_{N-1}=k_{N-2}+k_{N-1},\quad
 \alpha_N=-\tfrac12\sum_{i=1}^Nk_i.
\end{equation}
Define the bimultiplicative function $\eps$ on the generated lattice by
\begin{equation}\label{eq:epsilon}
 \eps(\alpha_i,\alpha_j)=
 \begin{cases}
 -1,&i=j,\\
 -1,&i<j\text{ and }\alpha_i+\alpha_j\in\Phi,\\
 +1,&\text{otherwise}.
 \end{cases}
\end{equation}
In particular, if $a=2\lambda$, set
$q_j=(a_j-a_N)/2$, $1\leq j\leq N-1$. The coefficients
$\lambda=\sum m_i\alpha_i$ are
\begin{equation}\label{eq:coefficients}
 \begin{split}
 m_i&=\sum_{j\leq i}q_j\quad(1\leq i\leq N-3),\\
 m_{N-2}&=\tfrac12\left(\sum_{j=1}^{N-2}q_j-q_{N-1}\right),\quad
 m_{N-1}=\tfrac12\sum_{j=1}^{N-1}q_j,\quad m_N=-a_N.
 \end{split}
\end{equation}
These coefficients are integral on the lattice described below.
Together, \eqref{eq:coefficients} and \eqref{eq:epsilon} give every
bracket sign in the chosen basis.
The standard EP sign identities include $\eps(\rho,\rho)=-1$
for $\rho\in\Phi$ and skewness when $\rho+\sigma\in\Phi$
\cite[Propositions 5.1--5.2]{epPublished}.

Let $\mc H\simeq\C^N$, write $h_v$ for the element associated
with $v$, and set $\g_n=\mc H\oplus\bigoplus_{\rho\in\Phi}\C x_\rho$.
The bracket is
\begin{equation}\label{eq:bracket}
 \begin{aligned}
 [h_u,h_v]&=0,&[h_u,x_\rho]&=(u,\rho)x_\rho,\\
 [x_\rho,x_{-\rho}]&=-h_\rho,&
 [x_\rho,x_\sigma]&=
 \begin{cases}\eps(\rho,\sigma)x_{\rho+\sigma},&\rho+\sigma\in\Phi,\\
 0,&\rho+\sigma\notin\Phi,\ \rho+\sigma\ne0.
 \end{cases}
 \end{aligned}
\end{equation}
The sum-zero case is always governed by the separate displayed rule.
Put
\begin{equation}\label{eq:hW}
 \h_n=\mc H\oplus\bigoplus_{\rho\in\Phi_O}\C x_\rho
 \simeq\so(2N,\C),\qquad
 W_n=\bigoplus_{\sigma\in\Phi_S}\C x_\sigma.
\end{equation}
Then $W_n$ is an irreducible half-spin $\h_n$-module and
$\g_n=\h_n\oplus W_n$ is $\Z_2$-graded. The orthogonal algebra
acts by derivations of the entire bracket
\cite[Proposition 6.1]{epPublished}. Equivalently, its elements belong to
the Jacobi nucleus: a Jacobiator with one entry in $\h_n$ vanishes.
We will use this nucleus property in the Jacobiator calculations
and in defining the tip product.

\subsection{The lattice and the locked families}
\begin{proposition}\label{prop:lattice}
The lattice generated by $\Phi$ is
\begin{equation}\label{eq:lattice}
 \Lambda_n=D_N^+=D_N\cup(s_0+D_N),\quad
 D_N=\{z\in\Z^N:\textstyle\sum z_i\in2\Z\},\quad
 s_0=\tfrac12\sum_i k_i.
\end{equation}
It is unimodular and is even exactly when $n$ is odd. For $n\geq2$,
its squared-length-two vectors are exactly $\Phi_O$.
\end{proposition}
\begin{proof}
Spinorial roots differ from $s_0$ by integral vectors of even coordinate
sum, while $2s_0\in D_N$. The indicated extension has index two in
$D_N$, whose determinant is four. It is integral because $4\mid N$;
its determinant is consequently one. Since $(s_0,s_0)=N/4=n+1$,
it is even exactly when $n+1$ is even. Every vector in the nonintegral
coset has squared norm at least $N/4$. For $n\geq2$ this exceeds
two, so every squared-length-two vector belongs to $D_N$, where the
claim follows from its integral coordinates.
\end{proof}

The locked $E_7$ family uses the intersection with
$v_{N-1}=v_N$, and the locked $E_6$ family the intersection with
$v_{N-2}=v_{N-1}=v_N$. Thus
\begin{equation}\label{eq:locked}
 \begin{array}{c|c|c|c}
 \text{family}&\text{rank}&\text{orthogonal roots}&\text{spinorial count}\\\hline
 E_7^{(n)}&N-1&D_{N-2}\sqcup\{\pm(k_{N-1}+k_N)\}&2^{N-2}\\
 E_6^{(n)}&N-2&D_{N-3}&2^{N-3}
 \end{array}
\end{equation}
In particular, the orthogonal-plus-Cartan subalgebra of the locked
families is respectively $\so(2N-4,\C)\oplus\slie_2(\C)$ and
$\so(2N-6,\C)\oplus\C$. Statements identifying that subalgebra
with $\so(2N,\C)$ refer only to $\g_n=E_8^{(n)}$.

\section{Jacobiators, simplicity and automorphisms}\label{sec:jacobi}
The Jacobiator is
\begin{equation}\label{eq:jacobi}
 J(x,y,z)=[[x,y],z]+[[y,z],x]+[[z,x],y].
\end{equation}
It is alternating and vanishes whenever one entry belongs to $\h_n$.
The following formulas provide the remaining cases.

\begin{proposition}[Spinorial Jacobiators]\label{prop:jacobi}
For $\alpha,\gamma\in\Phi_S$, $\gamma\ne\pm\alpha$,
\begin{equation}\label{eq:opposite}
 J(x_\alpha,x_{-\alpha},x_\gamma)=
 \left(-(\alpha,\gamma)+
 \mathbf1_{\gamma-\alpha\in\Phi}-
 \mathbf1_{\gamma+\alpha\in\Phi}\right)x_\gamma.
\end{equation}
For three distinct spinorial roots with no opposite pair, the
Jacobiator vanishes unless their total is a spinorial root and
exactly one pairwise sum is a root. In that remaining case,
if the pair is $(\alpha,\beta)$, it is
\begin{equation}\label{eq:translation}
 J(x_\alpha,x_\beta,x_\gamma)=
 \eps(\alpha,\beta)\eps(\alpha+\beta,\gamma)
 x_{\alpha+\beta+\gamma}.
\end{equation}
No Jacobiator has a Cartan component.
\end{proposition}
\begin{proof}
In the opposite-pair case, the first term is
$[-h_\alpha,x_\gamma]=-(\alpha,\gamma)x_\gamma$.
When present, the second and third terms have coefficients $-1$
and $+1$, respectively, by bimultiplicativity and
$\eps(\alpha,\alpha)=-1$. This gives \eqref{eq:opposite}.

For a triple with no opposite pair, every inner bracket, if nonzero,
is an orthogonal root vector. Its bracket with a spinorial vector
is either zero or a spinorial root vector, so a Cartan term cannot
occur. A term can survive only if the total
$\delta=\alpha+\beta+\gamma$ is spinorial. In doubled coordinates,
this means there is no coordinate where all three signs agree.
Partition the coordinates according to which of the three signs is
the exceptional one; call the sizes $m_\alpha,m_\beta,m_\gamma$.
They are positive: a zero part would make the other two roots
opposite. They are even: pairwise sign differences make the three
multiplicities have the same parity, and their sum $N$ is even. Also $m_\alpha+m_\beta+m_\gamma=N$.
The pair $\alpha+\beta$ is a root exactly when $m_\gamma=2$.
Three surviving terms would require $N=6$, which never occurs here.

If two terms survive, say the pairs $\alpha+\beta$ and
$\beta+\gamma$, then $m_\alpha=m_\gamma=2$ and
$(\alpha,\gamma)=n-1$. Put
$p=\eps(\alpha,\beta)\eps(\alpha+\beta,\gamma)$ and
$q=\eps(\beta,\gamma)\eps(\beta+\gamma,\alpha)$.
Bimultiplicativity gives
$q/p=-\eps(\gamma,\alpha)/\eps(\alpha,\gamma)$, since
$\alpha+\beta$ is a root. The spinorial commutator identity
\cite[Proposition 5.1]{epPublished} is
\begin{equation}\label{eq:epscomm}
 \eps(\alpha,\gamma)\eps(\gamma,\alpha)
 =(-1)^{(\alpha,\gamma)-(n-1)}.
\end{equation}
For this pair it equals $+1$, and the cyclic orientation gives
the cancellation explicitly:
\[
 q/p=\frac{\eps(\beta,\alpha)\eps(\gamma,\alpha)}
 {\eps(\alpha,\beta)\eps(\alpha,\gamma)}=-1.
\]
Thus the two terms cancel. With one term, \eqref{eq:translation}
is the bracket rule itself; with no term the value is zero.
\end{proof}

\begin{corollary}\label{cor:imJ}
$\g_n$ is Lie if and only if $n=1$. For $n\geq2$,
$\Span\{J(x,y,z):x,y,z\in\g_n\}=W_n$.
\end{corollary}
\begin{proof}
At $n=1$ the construction is $\mathfrak e_8(\C)$ \cite{adams,kollross}.
At $n\geq2$, for any $\gamma\in\Phi_S$ choose $\alpha\in\Phi_S$
differing from it in exactly $2n$ coordinates. Then
$(\alpha,\gamma)=1$, which is neither $n$ nor $-n$.
Formula~\eqref{eq:opposite} gives $J=-x_\gamma$.
All Jacobiators are in $W_n$, so these values span it.
Here and below ``image of the Jacobiator'' means its linear span,
or the image of the linear map $\Lambda^3\g_n\to\g_n$.
\end{proof}

\begin{theorem}\label{th:aut}
$\g_n$ is simple as an anticommutative complex algebra for every
$n\geq1$. For $n\geq2$,
\begin{equation}\label{eq:aut}
 \Der(\g_n)=\operatorname{ad}\h_n\simeq\so(2N,\C),\qquad
 \Aut(\g_n)\simeq\HSpin(2N,\C)=\operatorname{Spin}(2N,\C)/K,
\end{equation}
where $K$ is the order-two central kernel on the chosen half-spin module.
In particular, this automorphism group is connected.
\end{theorem}
\begin{proof}
The bracket satisfies $[W_n,W_n]=\h_n$. Indeed, opposite weights
give $h_\sigma$ and the spinorial weights span $\C^N$.
For any $\rho\in\Phi_O$, choose $\sigma\in\Phi_S$ with
$(\rho,\sigma)=1$; then $\rho-\sigma\in\Phi_S$ and its
bracket with $x_\sigma$ is a nonzero multiple of $x_\rho$.
An ideal is $\h_n$-stable. The two summands $\h_n$ and $W_n$
are nonisomorphic irreducible $\h_n$-modules, so complete
reducibility makes an ideal a sum of these summands. Neither
summand alone is an ideal, since $[\h_n,W_n]=W_n$ and
$[W_n,W_n]=\h_n$. Simplicity follows.

For $n\geq2$, $W_n$ is characteristic by Corollary~\ref{cor:imJ},
and $\h_n=[W_n,W_n]$ is characteristic too. Restrict a derivation
to $\h_n$, subtract the corresponding inner derivation, and apply
Schur's lemma on $W_n$. The remainder is a scalar $\lambda$;
applying it to $[W_n,W_n]$ gives $2\lambda[W_n,W_n]=0$,
hence $\lambda=0$.

To determine the full group, let an automorphism of $\g_n$ restrict
to $\h_n$. Since $N\geq12$, the only nontrivial diagram automorphism
of $D_N$ exchanges the two half-spin modules. It cannot preserve
the chosen $W_n$, so the restriction is inner. After a lift to $\operatorname{Spin}(2N,\C)$ is removed,
the remaining automorphism is the identity on $\h_n$ and a scalar
$\lambda$ on $W_n$, with $\lambda^2=1$. Both signs occur in the
spin action: the central $-1$ acts trivially on $\h_n$ and as
$-1$ on $W_n$. This proves the asserted full group and its
connectedness. Its dimension is $N(2N-1)$.
\end{proof}

\begin{remark}[Equivariance and the Fierz obstruction]\label{rem:fierz}
The adjoint representation occurs with multiplicity one in
$\Lambda^2W_n$: the Clifford decomposition of $W_n\otimes W_n$
contains the exterior two-form representation exactly once,
and its infinitesimal spin action is skew for the invariant
symmetric form on $W_n$ \cite{adams,abs}. Thus a nonzero equivariant map
$\pi:\Lambda^2W_n\to\h_n$ is unique up to scale. On
$\h_n\oplus W_n$, use the usual Lie bracket and module action
and set $[s,t]=c\pi(s,t)$. All nonzero values of $c$ give
isomorphic anticommutative algebras over $\C$ after rescaling
$W_n$. For $c=0$ one instead obtains the Lie semidirect product
$\h_n\ltimes W_n$. For general $c$, the spinorial Jacobiator is
\begin{equation}\label{eq:fierz}
 c\bigl(\pi(s,t)\cdot u+\pi(t,u)\cdot s+\pi(u,s)\cdot t\bigr).
\end{equation}
The scalar normalization is carried by $c$ and the chosen map
$\pi$. Within this EP family, the nonzero extension is Lie exactly
at $n=1$.
\end{remark}

\section{Tetracode gradings and their fibres}\label{sec:grading}
\subsection{Characters and the signed-pattern lemma}
Any character $\chi:\Lambda_n\to\F_3^2$ induces a grading
\begin{equation}\label{eq:grading}
 (\g_n)_0=\mc H\oplus\bigoplus_{\chi(\rho)=0}\C x_\rho,
 \qquad (\g_n)_c=\bigoplus_{\chi(\rho)=c}\C x_\rho\quad(c\ne0).
\end{equation}
The bracket rule directly gives
$[(\g_n)_c,(\g_n)_d]\subseteq(\g_n)_{c+d}$.
Since the indices of $D_N$ in both $\Lambda_n$ and $\Z^N$
are two, and two is invertible in $\F_3$, restriction and unique
extension give
\begin{equation}\label{eq:characters}
 \Hom(\Lambda_n,\F_3^2)\simeq(\F_3^2)^N,
 \quad \chi(\pm k_i\pm k_j)=\pm c_i\pm c_j,
 \quad \chi(\tfrac12\sum_i\epsilon_i k_i)=2\sum_i\epsilon_i c_i.
\end{equation}
Here $c_i$ is the value on $k_i$ of the extension to $\Z^N$;
the $k_i$ need not themselves belong to $\Lambda_n$.
We call such a grading \emph{balanced of Okubo type} if its
eight nonzero components are nonempty and have equal dimensions.
This terminology records the dimension pattern. It leaves the
multiplication unchanged and concerns character-induced gradings.
For the established Jordan gradings of exceptional Lie algebras,
see \cite{elduqueJordan}.

A \emph{tetracode character} has four nonzero coordinate values,
one representative of each line in $\mathbf P^1(\F_3)$, and all
other values zero. Let $C$ be its four-coordinate code block.
\begin{lemma}\label{lem:signed}
For any such representatives $u_1,u_2,u_3,u_4$, no signed sum
$\sum_j\epsilon_j u_j$ is zero. Each nonzero value has exactly
two preimages among the sixteen sign patterns; these two patterns
have opposite parity.
\end{lemma}
\begin{proof}
Changing signs, permuting columns and applying $GL_2(\F_3)$ reduces
the map to \eqref{eq:tetracode}. Its kernel is the tetracode;
a sign pattern has weight four, so cannot lie in that kernel.
Two distinct sign patterns in a fibre differ in exactly three
positions, because their difference is a nonzero codeword.
They therefore have opposite sign parity. A fibre contains at
most two patterns: among three patterns two have the same parity,
contradicting the preceding conclusion. Sixteen patterns are
distributed over eight nonzero values, so every value has two.
\end{proof}

\subsection{Ambient and locked fibre counts}
\begin{theorem}\label{th:ambientGrading}
A tetracode character gives
\begin{equation}\label{eq:ambientFibres}
 (\g_n)_0\simeq\so(2N-8,\C)\oplus\C^4,
 \qquad \dim(\g_n)_c=8n+3+2^{4n}=t_n\quad(c\ne0).
\end{equation}
\end{theorem}
\begin{proof}
Orthogonal roots supported in the zero block form $D_{N-4}$.
No root with both indices in $C$ has degree zero, since distinct
representatives are not proportional. For any nonzero degree,
roots with one index in $C$ and one outside it contribute
$2(N-4)=8n$ vectors. The 24 roots with both indices in $C$
contribute three to each degree, as a direct calculation with
the four columns of $H$ shows.

For spinorial roots the degree depends on four signs. Each nonzero
degree receives two patterns, each allowing $2^{N-5}$ choices
on the remaining coordinates after parity is imposed. Its
spinorial contribution is $2^{N-4}=2^{4n}$. No spinorial root
has degree zero. Adding the orthogonal contribution proves the
count. The Cartan has dimension $N$, four greater than the rank
of $D_{N-4}$, which gives the zero component.
\end{proof}
At $n=1$, the degree-zero roots form $D_4$. Both $D_4$ and the
idempotent system $A_2^4$ have 24 roots, but their ranks are four
and eight. This gives a simple way to distinguish the two uses
of the tetracode.

For the locked families, place $C$ among the coordinates not
locked in \eqref{eq:locked}. The same argument gives common
nonzero dimensions
\begin{equation}\label{eq:lockedFibres}
 t_n^{(7)}=8n-1+2^{4n-1},\qquad
 t_n^{(6)}=8n-3+2^{4n-2}.
\end{equation}
For $E_7^{(n)}$, the free orthogonal contribution is
$2(N-6)+3=8n-1$, and the locked $A_1$ has degree zero.
The spinorial count per nonzero degree is $2^{N-5}$.
For $E_6^{(n)}$, these numbers become $2(N-7)+3=8n-3$
and $2^{N-6}$. At $n=1$ the fibre dimensions are 15 and 9. The products and
nuclei of these locked families enter separately in their geometry
and automorphism groups.

\section{The tip product and an explicit Jordan obstruction}\label{sec:tip}
\subsection{The root census and frame}
Fix the external $A_2$ generated by $k_1-k_2$ and $k_2-k_3$.
Use the tip defined, in doubled coordinates, by
$a_1=a_2$ and $a_2-a_3=2$. It consists of
\begin{equation}\label{eq:tiproots}
 \{k_1+k_2\}\ \sqcup\ \{-k_3\pm k_j:4\leq j\leq N\}
 \ \sqcup\ \{\sigma\in\Phi_S:2\sigma_1=2\sigma_2=1,
 \ 2\sigma_3=-1\}.
\end{equation}
Its complex span $T_n$ has dimension
\begin{equation}\label{eq:tipdim}
 t_n=1+2(N-3)+2^{N-4}=8n+3+2^{4n}.
\end{equation}
There is no root sum between two tip roots, so $[T_n,T_n]=0$.
Choose the orthogonal frame
\begin{equation}\label{eq:frame}
 f_1=k_1+k_2,\qquad f_2=-k_3+k_4,\qquad f_3=-k_3-k_4.
\end{equation}
Put $E_i=x_{f_i}$ and use the following signs and normalization:
\begin{equation}\label{eq:tipproduct}
 I=E_1+E_2+E_3,\qquad
 I^-=-\sum_{i=1}^3x_{-f_i},\qquad
 x\circ y=\tfrac12[[x,I^-],y].
\end{equation}
We use the unital convention of \cite[Definition 8.4 and
equation (8.5)]{epPublished}; see also \cite[Definition 4.4 and
equation (4.5)]{ep2}. The frame \eqref{eq:frame} belongs to the
tip \eqref{eq:tiproots}.

\begin{proposition}\label{prop:unit}
Formula~\eqref{eq:tipproduct} defines a commutative unital
product on $T_n$. The $E_i$ are orthogonal idempotents and
\begin{equation}\label{eq:Peirce}
 E_i\circ x_\rho=\tfrac12(\rho,f_i)x_\rho.
\end{equation}
\end{proposition}
\begin{proof}
The external $A_2$ weight shows that a double bracket in
\eqref{eq:tipproduct}, when nonzero, belongs to the same tip.
Since $I^-\in\h_n$, the Jacobi nucleus identity and
$[T_n,T_n]=0$ imply commutativity. Orthogonality of the $f_i$
gives $[E_i,I^-]=h_{f_i}$; this yields
\eqref{eq:Peirce}, $E_i\circ E_j=\delta_{ij}E_i$ and
$I\circ x_\rho=x_\rho$, because every tip root has
$\sum_i(\rho,f_i)=2$.
\end{proof}
This construction is over $\C$. Every off-diagonal root vector
has square zero, so the real span of the root basis is not a
Euclidean Jordan algebra. At $n=1$ the product is the complex
Albert product. Its positive real form will be obtained from a
different real structure in Section~\ref{sec:cone}.

\subsection{Blocks and the compatible grading}
The simultaneous eigenspaces in \eqref{eq:Peirce} are
\begin{equation}\label{eq:blocks}
 T_n=\C E_1\oplus\C E_2\oplus\C E_3\oplus T_{23}
 \oplus T_{12}\oplus T_{13},\quad
 \dim T_{23}=8n,\quad \dim T_{12}=\dim T_{13}=2^{4n-1}.
\end{equation}
The tip trace is the linear functional that is one on each $E_i$
and vanishes on the three off-diagonal blocks. We denote it by
$\tr$, as for the matrix trace used earlier.
The roots $-k_3\pm k_j$ for $j\geq5$ have scalar products
$(0,1,1)$ with the frame. For a spinorial root let
$\epsilon_4=2\rho_4\in\{\pm1\}$; then
\begin{equation}\label{eq:spinFrame}
 (\rho,f_1)=1,\quad (\rho,f_2)=\tfrac{1+\epsilon_4}{2},
 \quad (\rho,f_3)=\tfrac{1-\epsilon_4}{2}.
\end{equation}
The remaining $N-4=4n$ signs identify the vector and two
opposite half-spin representations of $\so(8n,\C)$ supported
on coordinates $5,\ldots,N$. That algebra fixes $I^-$, hence
acts by derivations of the tip product. We use \emph{Peirce blocks} for these simultaneous eigenspaces;
the definition here uses only \eqref{eq:Peirce}.

If the tetracode block is contained in $\{4,\ldots,N\}$, direct
counting gives
\begin{equation}\label{eq:tipFibres}
 \dim(T_n)_0=8n-5,\qquad
 \dim(T_n)_c=1+2^{4n-3}\quad(c\ne0).
\end{equation}
The zero orthogonal count is $1+2(N-7)$; each nonzero degree
receives one orthogonal root. The two signed code patterns
for that degree have opposite parity, giving $2^{N-7}$
spinorial roots. On the $N-7$ coordinates outside
$\{1,2,3\}\cup C$, each nonzero fibre is, as a module,
\begin{equation}\label{eq:tipModule}
 (T_n)_c\simeq\mathbf1\oplus S^+\oplus S^-
 \quad\text{for }\so(2N-14,\C).
\end{equation}
For $N=8$, the last algebra is the one-dimensional abelian
$\so(2,\C)$, and the two spinor weights are interpreted accordingly.

For a grading of the \emph{product}, impose the stronger condition
\begin{equation}\label{eq:compatible}
 C\subseteq\{5,\ldots,N\}.
\end{equation}
Then all frame roots and $I^-$ have degree zero, and
$ (T_n)_c\circ(T_n)_d\subseteq(T_n)_{c+d}$.
The Lie algebra in \eqref{eq:tipModule} preserves the tip and
character. If it moves coordinate four, however, it moves $I^-$.
Its subalgebra fixing $I^-$ acts by derivations of the product.

At $n=1$, the dimensions in \eqref{eq:tipFibres} are $3+8\cdot3$.
They agree with a coarsening of the known fine $\Z_3^3$ grading
of the complex Albert algebra \cite{draperMartin,elduqueKochetov}.
Whether these gradings are conjugate requires a comparison of
their homogeneous products. At general levels, a uniform 27-component grading
would require
\begin{equation}\label{eq:div27}
 27\mid t_n\quad\Longleftrightarrow\quad n\equiv1\pmod{27}.
\end{equation}
To see this, $16^n$ has period nine modulo 27, so $t_n$ has
period 27; evaluation of its 27 residues gives only residue one.
Thus a uniform refinement is impossible for $2\leq n\leq27$.
Existence at $n\equiv1\pmod{27}$ and nonuniform refinements
remain separate questions.

\subsection{An identity-level non-Jordan witness}
\begin{theorem}\label{th:tipNonJordan}
The complex algebra $(T_n,\circ)$ is Jordan exactly at $n=1$.
\end{theorem}
\begin{proof}
At $n=1$ this is the Albert product
\cite[Remark 8.1]{epPublished}; see also \cite{marraniTruini}.
For $n\geq2$, put $m=N-4=4n\geq8$ and use doubled vectors
\begin{equation}\label{eq:JordanRoots}
 \begin{split}
 a&=(1,1,-1,1,\underbrace{1,\ldots,1}_{m-1},-1),\\
 b&=(1,1,-1,-1,-1,-1,-1,
       \underbrace{1,\ldots,1}_{m-4},-1),\\
 a'&=2f_1+2f_2-a,\qquad b'=2f_1+2f_3-b.
 \end{split}
\end{equation}
All four are doubled spinorial tip roots. Let $A=a/2$, $A'=a'/2$,
$B=b/2$, $B'=b'/2$. The bracket signs give
\[
 x_A\circ x_{A'}=\tfrac{\sigma_A}{2}(E_1+E_2),\qquad
 x_B\circ x_{B'}=\tfrac{\sigma_B}{2}(E_1+E_3),
 \qquad \sigma_A,\sigma_B\in\{\pm1\}.
\]
The two coefficients on each diagonal pair are equal by
\eqref{eq:epsilon}; their signs can also be evaluated directly
using \eqref{eq:coefficients}. Each individual root vector has
square zero. Every cross product between $\{x_A,x_{A'}\}$ and
$\{x_B,x_{B'}\}$ is zero: on the last $m$ coordinates the
relevant sign patterns agree in either three or $m-3$ positions,
whereas a surviving orthogonal output would require agreement
in one position. This is a root-membership calculation and
does not depend on the value of $\sigma_A$ or $\sigma_B$.

Set $u=x_A+\sigma_Ax_{A'}$ and
$w=x_B+\sigma_Bx_{B'}$. Thus
$u^2=E_1+E_2$, $w^2=E_1+E_3$, and $u\circ w=0$.
For $X=u+w$, $Y=E_2$, the eigenvalues in
\eqref{eq:Peirce} give
\begin{equation}\label{eq:Jordanator}
 (X^2\circ Y)\circ X-X^2\circ(Y\circ X)
 =\tfrac12u-\tfrac34u=-\tfrac14u\ne0.
\end{equation}
This violates the Jordan identity.
\end{proof}
The witness is constructed in the complex root algebra itself.
For the basis convention used by the exact checker, the signs
are $\sigma_A=1$ and $\sigma_B=-1$ at $n=2,3$.

\section{Unscaled root-frame obstructions}\label{sec:arithmetic}
Throughout this section, $A_2$ denotes the root lattice with
Gram matrix $\left(\begin{smallmatrix}2&-1\\-1&2\end{smallmatrix}\right)$.
An orthogonal frame of $m$ copies has rank $2m$. The scale is fixed by this Gram matrix. We discuss scaled
lattices $A_2(c)$ after the theorem.
\begin{theorem}\label{th:frameNo}
For $n\geq2$, an orthogonal $A_2^m$ frame in $\Lambda_n$
has $m\leq\lfloor N/3\rfloor$. In particular, none has full rank.
\end{theorem}
\begin{proof}
By Proposition~\ref{prop:lattice}, its roots belong to $\Phi_O$.
If $\alpha,\beta\in\Phi_O$ have $(\alpha,\beta)=-1$,
their two-coordinate supports meet in exactly one position.
The $A_2$ therefore lies in a three-coordinate block. The
orthogonal complement of its two-dimensional plane inside
that block is a line generated by a vector with all three
coordinates nonzero.

An orthogonal $D_N$ root has at most two nonzero coordinates.
Its projection to the first block must lie on that complement
line, and so must be zero. Consequently every root of a
second orthogonal $A_2$ avoids the first block. Iterating gives
disjoint blocks and $3m\leq N$. Full rank would require
$m=N/2$, a contradiction.
\end{proof}
At $n=1$, \eqref{eq:a2frame} provides the full-rank frame.
Its index is nine, and its glue is the tetracode. More generally,
if an unscaled full-rank $A_2^{N/2}$ frame lay in a unimodular
lattice, determinant comparison would require index $3^{N/4}$.
The theorem prevents this configuration here for every $n\geq2$.

\begin{remark}[Scaled frames]\label{rem:scaled}
The obstruction does not rule out orthogonal frames of scaled
$A_2$ lattices. Multiply an $A_2^4$ root frame in the coordinate
$E_8$ model by two. Its vectors now belong to $D_8$, have
squared length eight and mutual scalar product $-4$ in each
pair; it is an $A_2(4)^4$ frame. For every odd $n\geq3$,
$8\mid N$, and blockwise copies give a full-rank
$A_2(4)^{N/2}$ frame in $D_N\subset\Lambda_n$.
The supplement gives an integer certificate. Full-rank scaled
frames therefore exist at these levels.
\end{remark}

For the locked families, the norm-two shell gives an even
simpler rank obstruction.
\begin{proposition}\label{prop:e6frame}
Let $\Lambda_n^{(6)}$ be the lattice generated by the locked
$E_6^{(n)}$ roots. It has a full-rank unscaled $A_2$ frame
exactly at $n=1$, where $A_2^3\subset E_6$ has index three.
For the locked $E_7^{(n)}$ lattice, such a full-rank frame
is impossible at every level because its rank is odd.
\end{proposition}
\begin{proof}
Put $M=N-3$. The $E_6^{(n)}$ lattice is generated by $D_M$
on the first $M$ coordinates and $s_0$; all locked spinorial
roots differ from $s_0$ by that subgroup and multiples of
$2s_0$. Its integral vectors have a common integer last
coordinate $\ell$ and first-coordinate sum congruent to
$\ell$ modulo two. If $\ell\ne0$, their norm is at least four
when $|\ell|=1$ (three locked squares and at least one free
square), and at least twelve when $|\ell|\geq2$.
If $\ell=0$, the squared-length-two vectors are exactly
$D_M$ roots. A half-integral vector has norm at least
$N/4=n+1$, which exceeds two at $n\geq2$.
All unscaled $A_2$ roots then lie in a subspace of rank
$M=N-3$, smaller than the lattice rank $N-2$; no full-rank
frame can occur. The level-one trinification lattice is
classical. The $E_7$ assertion follows from rank $N-1$.
\end{proof}
Above level one the norm-two shell spans $D_N$ and leaves out
the spinorial coset. Integral constructions using that coset,
as well as scaled frames and ternary gradings, can still be studied.

\section{A positive real Clifford model and its cone}\label{sec:cone}
We now build the positive real Clifford model $H_n^{\rm Cl}$.
Theorem~\ref{th:identification} matches its complexification with
the root tip $(T_n,\circ)$, giving the published $H$-algebra
realization \cite[Section 8]{epPublished} in our conventions.

\subsection{Clifford blocks and their multiplication}
Put $d=8n$ and $b=2^{4n-1}$. Let $V=\R^d$ and let
$S=S^+\oplus S^-$ be a minimal graded real module for
$\Cl(0,d)$, with $\dim S^+=\dim S^-=b$ and positive
invariant inner products \cite{abs}. Clifford multiplication has block form
\begin{equation}\label{eq:clifford}
 \Gamma(v)=\begin{pmatrix}0&-C(v)^{\mathsf T}\\ C(v)&0\end{pmatrix},
 \quad \Gamma(v)\Gamma(w)+\Gamma(w)\Gamma(v)
 =-2\langle v,w\rangle I_S,
\end{equation}
where $C(v):S^+\to S^-$. In particular,
$C(v)^{\mathsf T}C(v)=\|v\|^2I$ and
$C(v)C(v)^{\mathsf T}=\|v\|^2I$.
Define $\beta:S^+\times S^-\to V$ by
\begin{equation}\label{eq:beta}
 \langle\beta(s,t),v\rangle=\langle C(v)s,t\rangle.
\end{equation}

The underlying space is
\begin{equation}\label{eq:Hcl}
 H_n^{\rm Cl}=\R^3\oplus S^+\oplus S^-\oplus V,
 \qquad X=\begin{pmatrix}x&s&t\\s^*&y&v\\t^*&v^*&z\end{pmatrix}.
\end{equation}
The symbol $*$ denotes the metric dual block. The upper block
product is $s\cdot v=C(v)s$; the remaining block products use
its metric adjoints and inner products. We follow the orientation
$A_{12}=S^+$, $A_{23}=V$, $A_{13}=S^-$ fixed in \cite{amsErr}. The isometry
$\|C(v)s\|^2=\|v\|^2\|s\|^2$ supplies the rank-three
metric nilalgebra, and hence the Vinberg construction
\cite{ac,cubics}. The strictly upper algebra is associative
because a product of three strictly upper matrices is zero;
the full Hermitian multiplication need not be associative.

For the Hermitian product $X\odot Y=\frac12(XY+YX)$, it is
convenient to give the square first. Its diagonal entries are
\begin{equation}\label{eq:clSquareDiag}
 x^2+\|s\|^2+\|t\|^2,\quad
 y^2+\|s\|^2+\|v\|^2,\quad
 z^2+\|t\|^2+\|v\|^2,
\end{equation}
and upper off-diagonal entries
\begin{equation}\label{eq:clSquareOff}
 (x+y)s+C(v)^{\mathsf T}t,\quad
 (x+z)t+C(v)s,\quad
 (y+z)v+\beta(s,t).
\end{equation}
Polarization defines $X\odot Y$. Its unit is $I=(1,1,1,0,0,0)$,
and its frame blocks have dimensions $b,b,d$. The cubic is
\begin{equation}\label{eq:cubic}
 D(X)=xyz-x\|v\|^2-y\|t\|^2-z\|s\|^2
       +2\langle C(v)s,t\rangle.
\end{equation}
The block order fixes the cubic and the Cholesky formulas below.
Permuting the blocks also changes the coordinate expressions of
the cone and its dual invariant.
The trace is $\tr X=x+y+z$. With powers associated as
$X^3=X\odot(X\odot X)$, a direct expansion gives
\begin{equation}\label{eq:newton}
 D(X)=\tfrac16\bigl((\tr X)^3-3\tr X\tr(X^2)+2\tr(X^3)\bigr).
\end{equation}
The identity uses the indicated association of $X^3$. Higher
powers require their own association conventions.

\subsection{The normalized root-tip identification}
The published paper uses the tip $T_{(1,1)}$ with frame roots
$p_1=k_1+k_N$, $p_2=k_1-k_N$, $p_3=-k_2-k_3$
\cite[(8.1)--(8.5)]{epPublished}; our tip is $T_{(0,2)}$.
The signed coordinate permutation
\begin{equation}\label{eq:publishedTipMap}
 \begin{split}
 k_1&\longmapsto-k_3,\quad k_2\longmapsto-k_1,\quad
 k_3\longmapsto-k_2,\quad k_4\longmapsto-k_5,\\
 k_j&\longmapsto k_{j+1}\ (5\leq j\leq N-1),\qquad
 k_N\longmapsto k_4
 \end{split}
\end{equation}
has four sign changes and belongs to the Weyl group of $D_N$.
It maps $(p_1,p_2,p_3)$ to $(f_2,f_3,f_1)$ and maps the
published vector block $12$ to our vector block $23$.
An inner Spin lift, followed by a Cartan torus adjustment of the
three frame root vectors, gives a pointed algebra isomorphism.
The lift supplies the scalar factors needed to preserve the
$\eps$ coefficients. Its action is available because the
orthogonal algebra lies in the Jacobi nucleus.

\begin{theorem}[Normalized published $H$-algebra realization]
\label{th:identification}
For every $n\geq1$ there is a complex algebra isomorphism
\begin{equation}\label{eq:tipIso}
 \varphi_n:(T_n,\circ)\longrightarrow
 (H_n^{\rm Cl}\otimes_\R\C,\odot)
\end{equation}
that preserves the ordered frame and trace form. In particular,
the Newton cubic of the root tip maps to \eqref{eq:cubic}.
The antilinear involution
$\kappa_n(X)=\varphi_n^{-1}(\overline{\varphi_n(X)})$
has a fixed real algebra $T_n^+$ isomorphic to $H_n^{\rm Cl}$,
with positive trace form. On that real form
\begin{equation}\label{eq:tipConeTransport}
 \Omega_n=\varphi_n^{-1}(\Omega_n^{\rm Cl})
\end{equation}
is a homogeneous convex cone. The realization depends on the chosen isometry and intertwiner.
Compatibility with an integral order or a character grading is
a separate problem.
\end{theorem}
\begin{proof}
Let $B_T(a,b)=\tr(a\circ b)$. It is associative by
\cite[Proposition 8.8 and Proposition A.2]{epPublished},
transported by \eqref{eq:publishedTipMap}. The frame eigenvalues
and root membership give
\[
 T_n=\bigoplus_{i=1}^3\C E_i\oplus S^+\oplus S^-\oplus V,
 \qquad S^+=T_{12},\quad S^-=T_{13},\quad V=T_{23}.
\]
For $U,W\in T_{ij}$,
\begin{equation}\label{eq:rootBlockForms}
 U\circ W=q_{ij}(U,W)(E_i+E_j),\qquad
 q_{ij}=\tfrac12 B_T|_{T_{ij}}.
\end{equation}
Each root $r$ in $T_{ij}$ pairs with $r'=f_i+f_j-r$:
$x_r\circ x_{r'}=\frac{\sigma_r}{2}(E_i+E_j)$,
$\sigma_r\in\{\pm1\}$; all other same-block root pairs vanish.
This follows by the bracket rule and bimultiplicativity of $\eps$,
as in \cite[Appendix A]{epPublished}.
Thus the three forms are symmetric and nondegenerate.
Write $q_+=q_{12}$, $q_-=q_{13}$ and $q_V=q_{23}$.
The Lie algebra $\so(V,q_V)\simeq\so(8n,\C)$ fixes the frame,
acts by product derivations, and acts on $S^\pm$ as the
two irreducible half-spin representations.

Define from the root product
\begin{equation}\label{eq:rootCliffordMaps}
 C_T(v)s=2v\circ s,\qquad C_T^\dagger(v)t=2v\circ t,\qquad
 \beta_T(s,t)=2s\circ t.
\end{equation}
Root membership places these values in $S^-$, $S^+$ and $V$,
respectively. Associativity of $B_T$ yields
\begin{equation}\label{eq:rootAdjoint}
 q_-(C_T(v)s,t)=q_+(s,C_T^\dagger(v)t)
             =q_V(v,\beta_T(s,t)).
\end{equation}
The maps are equivariant. The Clifford/exterior-algebra
realization of a half-spin endomorphism module consists
of even exterior powers of $V$, with one middle-degree summand.
It contains no trace-free symmetric two-tensor representation.
Hence
\[
 \dim\Hom_{\so(V)}(\operatorname{Sym}^2 V,\operatorname{End}S^\pm)=1.
\]
It follows that
\[
 C_T^\dagger(v)C_T(w)+C_T^\dagger(w)C_T(v)
       =2c_+q_V(v,w)I_{S^+},
\]
and similarly on $S^-$, with scalar $c_-$.

A short root calculation fixes both constants.
Put $v_+=-k_3+k_5$, $v_-=-k_3-k_5$, $a=-k_4+k_5$,
and choose a spinorial root $r\in T_{12}$ with coordinates
$r_4=+1/2$, $r_5=-1/2$; let $t=r+a\in T_{13}$.
Write $x_{v_+}\circ x_{v_-}
=\frac{\sigma_V}{2}(E_2+E_3)$.
The only surviving frame terms give
\[
 C_T(x_{v_+})x_r
   =-\eps(v_+,f_2)\eps(a,r)x_t,\qquad
 C_T^\dagger(x_{v_-})x_t
   =-\eps(v_-,f_3)\eps(-a,t)x_r .
 \]
Since $\eps(a,t)=-\eps(a,r)$ and
$\eps(a,v_-)=\eps(v_-,a+v_-)=\eps(v_-,f_3)$,
the product of these two coefficients is
$-\eps(v_+,f_2)\eps(a,v_-)=\sigma_V$.
The reversed term on $x_r$ vanishes by root membership,
and $2q_V(x_{v_+},x_{v_-})=\sigma_V$.
Thus $c_+=1$; the reversed calculation on $x_t$ gives $c_-=1$.
Consequently
\begin{equation}\label{eq:rootCliffordNormalized}
 \Gamma_T(v)=
 \begin{pmatrix}0&-C_T^\dagger(v)\\C_T(v)&0\end{pmatrix},
 \qquad
 \Gamma_T(v)\Gamma_T(w)+\Gamma_T(w)\Gamma_T(v)
       =-2q_V(v,w)I.
\end{equation}

Choose a complex isometry of $(V,q_V)$ with the
complexification of $\R^{8n}$. The module $S^+\oplus S^-$
has dimension $2^{4n}$, so \eqref{eq:rootCliffordNormalized}
is the irreducible module of the complex Clifford algebra.
The classification of Clifford modules \cite{abs} identifies it,
with the chosen chirality convention, with the complexification
of the real graded module used in \eqref{eq:clifford}.
The form $q_+\oplus q_-$ makes all $\Gamma_T(v)$ skew-adjoint.
Two such invariant nondegenerate forms differ by a scalar:
their ratio commutes with the irreducible Clifford action.
A common complex rescaling of the spinorial intertwiner therefore
makes it an isometry on both chiral summands.
Equation \eqref{eq:rootAdjoint} then matches every mixed map,
including $\beta_T$, and \eqref{eq:rootBlockForms} matches
the diagonal squares. Together with the frame eigenvalues,
polarization proves \eqref{eq:tipIso}.

The real model has positive trace form \eqref{eq:tracePairing}.
Transport of ordinary complex conjugation supplies $\kappa_n$,
its positive fixed algebra and \eqref{eq:tipConeTransport}.
The Newton trace formula proves the cubic assertion.
This gives the published identification with the normalization
of \eqref{eq:tipproduct}.
\end{proof}

\begin{remark}[Real form and induced tip derivations]\label{rem:identification}
The root-vector basis spans a split real coordinate algebra,
whose same-block trace forms pair isotropic vectors. The positive
algebra $T_n^+$ uses the transported conjugation instead. Thus
the root-coordinate grading and the positive real structure are
separate data; their compatibility, including any integral
compatibility, remains to be determined.

The published result
\cite[Theorem 11.4]{epPublished} identifies the derivations of
the tip induced by its ambient Magic Star centre with
$\mathfrak b_{4n}=\so(8n+1,\C)$ for $n>1$.
At $n=1$ they enhance to $\mathfrak f_4$.
This is stronger than the frame-fixing $\so(8n,\C)$ action.
Here the derivations are those induced from the ambient centre.
Determining all abstract derivations of the tip would require
an additional argument.
\end{remark}

\subsection{Cholesky coordinates, convexity and positivity}
For an upper triangular element
\[
 A=\begin{pmatrix}a&b_0&c_0\\0&d_0&e_0\\0&0&f\end{pmatrix},
 \qquad a,d_0,f>0,\ b_0\in S^+,\ c_0\in S^-,\ e_0\in V,
\]
the coordinates of $X=AA^*$ are
\begin{equation}\label{eq:choleskyForward}
 \begin{aligned}
 x&=a^2+\|b_0\|^2+\|c_0\|^2,&y&=d_0^2+\|e_0\|^2,&z&=f^2,\\
 s&=d_0b_0+C(e_0)^{\mathsf T}c_0,&t&=fc_0,&v&=fe_0.
 \end{aligned}
\end{equation}
Substitution in \eqref{eq:cubic}, using the Clifford isometry,
yields $D(AA^*)=(ad_0f)^2$. Conversely,
\begin{equation}\label{eq:choleskyInverse}
 \begin{split}
 f^2&=z,\qquad d_0^2=(yz-\|v\|^2)/z,\qquad
 a^2=D(X)/(yz-\|v\|^2),\\
 e_0&=v/f,\qquad c_0=t/f,\qquad
 b_0=(s-z^{-1}C(v)^{\mathsf T}t)/d_0.
 \end{split}
\end{equation}
Consequently the cone is explicitly
\begin{equation}\label{eq:cone}
 \Omega_n^{\rm Cl}=
 \{X:z>0,\ yz-\|v\|^2>0,\ D(X)>0\}=
 \{AA^*:a,d_0,f>0\}.
\end{equation}
Vinberg's construction makes this an open homogeneous convex
cone with pointed closure \cite{vinberg,ac,cubics}. The inverse
\eqref{eq:choleskyInverse} singles out its positive branch: all
three diagonal Cholesky parameters are positive.

The associated triangular Vinberg group acts simply transitively
on the cone, with dimension $\dim H_n^{\rm Cl}=t_n$.
The cubic is a relative invariant, transformed by the square
of the triangular determinant. The determinant-one subgroup
therefore acts transitively on
\begin{equation}\label{eq:realManifold}
 \mc M_{5,n}=\{X\in\Omega_n^{\rm Cl}:D(X)=1\}.
\end{equation}
At the identity the full logarithmic Hessian is
\begin{equation}\label{eq:hessian}
 -\mathrm{Hess}_I\log D(\delta X,\delta X)
 =\delta x^2+\delta y^2+\delta z^2
 +2\|\delta s\|^2+2\|\delta t\|^2+2\|\delta v\|^2.
\end{equation}
Relative invariance transports this positive form to every
point of the cone. Its restriction to $D=1$ is thus a
positive projective special real metric. All these cones have Vinberg rank three. Increasing $n$ enlarges
the Clifford module while leaving this rank unchanged. For
higher-rank special $T$-algebras, see \cite{rank4}.

\subsection{The Fierz defect, adjoints and the dual cone}
Use the trace pairing
\begin{equation}\label{eq:tracePairing}
 B(X,Y)=xx'+yy'+zz'+2\langle s,s'\rangle
 +2\langle t,t'\rangle+2\langle v,v'\rangle.
\end{equation}
The quadratic gradient $X^\#=\nabla_BD(X)$ has diagonal
entries $yz-\|v\|^2$, $xz-\|t\|^2$, $xy-\|s\|^2$ and
off-diagonal entries
\begin{equation}\label{eq:adjoint}
 C(v)^{\mathsf T}t-zs,\qquad C(v)s-yt,\qquad \beta(s,t)-xv.
\end{equation}
It also equals the quadratic Newton expression
$X^2-(\tr X)X+\frac12((\tr X)^2-\tr(X^2))I$.
Define the nonnegative defect
\begin{equation}\label{eq:defect}
 \Delta(s,t)=\|s\|^2\|t\|^2-\|\beta(s,t)\|^2\geq0.
\end{equation}
For fixed $s$, the vectors $C(e_i)s$ form an orthogonal
family of squared norm $\|s\|^2$ in $S^-$. Bessel's inequality
proves the last assertion. At $n=1$, $b=d=8$ and that
family spans $S^-$ whenever $s\ne0$, so $\Delta\equiv0$.
At $n\geq2$, $b>d$; unit $s$ and unit
$t\perp C(V)s$ give $\beta(s,t)=0$ and $\Delta=1$.

This also gives an explicit non-Jordan witness for the real
Clifford model. Let $U$ have only the $12$ block $s$ nonzero
and $W$ only the $13$ block $t$ nonzero, with those unit choices.
Then $U^2=E_1+E_2$, $W^2=E_1+E_3$ and $U\odot W=0$.
The calculation in \eqref{eq:Jordanator} applies and yields $-U/4$.
For $X=U+W$, its adjoint is diagonal with entries $(0,-1,-1)$,
and hence
\begin{equation}\label{eq:adjointFailure}
 (X^\#)^\#=E_1\ne0=D(X)X.
\end{equation}
The quadratic Newton adjoint therefore fails the Jordan adjoint
identity above level one.

Under the pairing \eqref{eq:tracePairing}, the dual cone and
its normalized degree-three invariant are
\begin{equation}\label{eq:dual}
 \begin{split}
 D^*(X)&=D(X)+\frac{\Delta(s,t)}x,\\
 (\Omega_n^{\rm Cl})^*&=
 \{X:x>0,\ xy-\|s\|^2>0,\ D^*(X)>0\}.
 \end{split}
\end{equation}
These are the corrected special Vinberg formulas
\cite{amsErr,ams}. At $n\geq2$, $D^*$ is genuinely rational:
$\Delta$ is a nonzero polynomial independent of $x$, so it
cannot be divisible by $x$. It is smooth on the dual cone.
It is homogeneous of degree three. The distinction between a
polynomial and a rational dual invariant is the correction
made in \cite{amsErr}.

The cubic and rational dual satisfy the logarithmic duality identities
\begin{equation}\label{eq:logDual}
 D^*(\nabla_B\log D(X))=D(X)^{-1},\qquad
 \nabla_B\log D^*(\nabla_B\log D(X))=X
 \quad(X\in\Omega_n^{\rm Cl}).
\end{equation}
To see this, let $g$ be the linear action of an element of the
triangular group, let $g^\dagger$ be its adjoint for $B$, and write
$D(gX)=\chi(g)D(X)$. The dual relative invariant obeys
$D^*(g^{-\dagger}Y)=\chi(g)^{-1}D^*(Y)$
\cite{amsErr,cubics}. The chain rule gives
\[
 \nabla_B\log D(gX)=g^{-\dagger}\nabla_B\log D(X),\qquad
 \nabla_B\log D^*(g^{-\dagger}Y)
       =g\nabla_B\log D^*(Y).
\]
At $I$ both gradients equal $I$ and both invariants equal one.
Since every $X$ in the cone is $gI$, substitution proves both
identities in \eqref{eq:logDual}. In particular, the first
gradient belongs to the dual cone. This argument uses $D^*$
on the positive dual domain, where it is smooth.

\subsection{Symmetry and boundary geometry}
\begin{theorem}\label{th:coneSymmetric}
The specified cone $\Omega_n^{\rm Cl}$ is symmetric if and only
if $n=1$, and the same holds for its transported cone $\Omega_n$.
At $n=1$ it is the positive Albert cone. The
projectivization of its nonzero rank-one boundary rays is
the compact Cayley plane $\mathbb O P^2$.
\end{theorem}
\begin{proof}
The classification of special rank-three Vinberg cones says
that a cone from the positive Clifford construction with
$\dim S^+=\dim S^-$ is symmetric exactly when the three
off-diagonal metric spaces have the same dimension
\cite[Corollary 4.2]{cubics}; equivalently its metric
nilalgebra is the one associated with a Euclidean rank-three
Jordan algebra \cite{jvnw}. Here the equality is $8n=2^{4n-1}$,
which holds at $n=1$ and fails thereafter: the right-hand
side is multiplied by 16 at each step while the left-hand
side grows linearly. At equality, real triality identifies
the three eight-dimensional spaces and gives the Albert
cone. The boundary statement is the standard Albert
rank-one description \cite{faraut,springer}.
\end{proof}
This proof uses the classification of the constructed Vinberg
cone. Koecher--Vinberg associates a Euclidean Jordan algebra
with a symmetric cone \cite{koecher,vinberg,faraut}; identifying
its product with a given product is a further step. For the
root tip, Theorem~\ref{th:identification} supplies that step.

The homogeneous tube $H_n^{\rm Cl}+i\Omega_n^{\rm Cl}$ carries
the positive special K\"ahler metric described in
Section~\ref{sec:supergravity}. Above level one, the classical
Albert construction of a projective plane is no longer available.
The boundary strata and their pairing with the dual cone remain
natural geometric objects. An incidence plane built from them
would need a definition of points and lines and a proof of the
incidence axioms; nonsymmetry of the cone settles neither its
existence nor the existence of a polarity.

\section{The level-one comparison}
\begin{theorem}\label{th:levelone}
For the objects defined in this paper, the following conditions
on $n\geq1$ are equivalent:
\begin{enumerate}
\item $\g_n$ satisfies the Jacobi identity;
\item all generalized roots have the same length;
\item $\Lambda_n$ contains a full-rank orthogonal unscaled $A_2$ frame;
\item the complex root-tip product $(T_n,\circ)$ is Jordan;
\item the transported positive cone $\Omega_n$, equivalently $\Omega_n^{\rm Cl}$, is symmetric;
\item $n=1$.
\end{enumerate}
At $n=1$ the frame is $A_2^4\subset E_8$, its index is nine,
and its glue code is the tetracode.
\end{theorem}
\begin{proof}
The root-length equality is $n+1=2$ by Lemma~\ref{lem:length}.
The remaining equivalences with $n=1$ are respectively
Corollary~\ref{cor:imJ}, Theorem~\ref{th:frameNo} and
\eqref{eq:a2frame}, Theorem~\ref{th:tipNonJordan}, and
Theorem~\ref{th:coneSymmetric}.
\end{proof}
The equivalences use the positive real form of
Theorem~\ref{th:identification} and the norm-two scale of the
lattice. Triality selects the same level: among $\so(8n)$,
only $\so(8)$ has an outer automorphism permuting its vector
and both half-spin representations. Balanced tetracode gradings
behave differently and persist for all $n$.

The locked families retain the fibre counts
\eqref{eq:lockedFibres} and the arithmetic statements of
Proposition~\ref{prop:e6frame}. The published paper also describes
$H$-algebras for the locked families \cite[Section 8 and Remark 8.3]{epPublished}.
Theorem~\ref{th:identification} treats $E_8^{(n)}$. Extending its
positive-real construction to the locked families requires their
real Clifford modules and product normalizations. Their different
nuclei likewise enter a separate automorphism calculation.

\section{Relation to theories with eight real supercharges}\label{sec:supergravity}
The cubic \eqref{eq:cubic}, restricted to its positive domain,
defines the supergravity models discussed here. By
Theorem~\ref{th:identification} it is the Newton cubic of the
EP tip on the chosen positive real form. We use $\mc N=2$
for eight real supercharges in four and five dimensions.
Spacelike reduction to three dimensions preserves these eight
supercharges, usually denoted $\mc N=4$. We recall the known
homogeneous models and give their relation to the tip product.

\subsection{Identification with the classified homogeneous series}
Set $y=u+w$, $z=u-w$ and $\Sigma=(s,t)\in S^+\oplus S^-$.
The cubic becomes
\begin{equation}\label{eq:dwpCubic}
 D=x(u^2-w^2-\|v\|^2)-u\|\Sigma\|^2
       +w(\|s\|^2-\|t\|^2)+2\langle C(v)s,t\rangle.
\end{equation}
The real symmetric matrices
\begin{equation}\label{eq:positiveGamma}
 \gamma_0=\begin{pmatrix}I&0\\0&-I\end{pmatrix},\qquad
 \gamma_i=\begin{pmatrix}0&C(e_i)^{\mathsf T}\\C(e_i)&0\end{pmatrix}
 \quad(1\leq i\leq8n)
\end{equation}
satisfy $\gamma_A\gamma_B+\gamma_B\gamma_A=2\delta_{AB}I$.
With $\eta=(w,v)$, the last two terms of
\eqref{eq:dwpCubic} are
\[\sum_A\eta^A\Sigma^{\mathsf T}\gamma_A\Sigma.\]
This is, after conventional linear rescalings and an overall
normalization, the de Wit--Van Proeyen Clifford cubic with
$q=8n$ and one irreducible real $\Cl(8n+1,0)$ module
\cite{dwp,dwvp}. In their notation it is
\begin{equation}\label{eq:Lseries}
 L(8n,1,0),\qquad \dim\Sigma=2^{4n},\qquad
 \dim\mc M_{5,n}=8n+2+2^{4n}=t_n-1.
\end{equation}
The two Clifford chiralities can be interchanged by the
corresponding convention change; the multiplicities are
$(P,\dot P)=(1,0)$ here. The level $n=1$ gives the octonionic magical model. Higher
levels select homogeneous nonsymmetric members of the same
classification.

\subsection{Five dimensions: cubic couplings and the positive branch}
Choose linear coordinates $h^I$, $1\leq I\leq t_n$, on
$H_n^{\rm Cl}$ and define the constant symmetric tensor
$C_{IJK}$ by $D(h)=C_{IJK}h^Ih^Jh^K$. We absorb the usual factorial into $C_{IJK}$.
The scalar fields of ungauged five-dimensional
Maxwell--Einstein supergravity lie on
$D(h)=1$ inside $\Omega_n^{\rm Cl}$, with
\begin{equation}\label{eq:5dmetric}
 a_{IJ}(h)=-\tfrac13\partial_I\partial_J\log D(h)\big|_{D=1},
 \qquad g_{ab}=\tfrac32a_{IJ}\partial_a h^I\partial_b h^J.
\end{equation}
An overall rescaling of the scalar metric is convention dependent;
the displayed choice is consistent with the usual
$h_a^I=-\sqrt{3/2}\,\partial_a h^I$ convention.
There are $t_n$ vector fields including the graviphoton and
$n_V^{(5)}=t_n-1$ vector multiplets. The same $C_{IJK}$
fixes the abelian Chern--Simons couplings
\cite{gst,dwp}. Equation~\eqref{eq:hessian} and relative invariance give the
positivity of the kinetic metrics on this branch.

Compact $\operatorname{Spin}(8n)$ acts simultaneously on
$v,s,t$ and preserves the cubic. The positive Clifford model
also has the familiar compact $\operatorname{Spin}(8n+1)$
action: $(w,v)$ is its vector and $\Sigma$ its spinor.
This corresponds to the $\mathfrak b_{4n}$ induced tip
derivations of \cite[Theorem 11.4]{epPublished}; at $n=1$
the Albert algebra has the larger compact $F_4$ symmetry.
In the classified homogeneous scalar model the Lorentz
factor associated with $(u,w,v)$ is locally $SO(8n+1,1)$,
and further transformations complete the known scalar symmetries
\cite{dwvp}. These groups act on the real cubic and scalar manifold.
The group $\Aut(\g_n)=\HSpin(8n+8,\C)$ for $n\geq2$ instead
preserves the ambient complex EP bracket.

\subsection{Four dimensions: the \texorpdfstring{$r$}{r}-map and special K\"ahler geometry}
Circle reduction gives $n_V^{(4)}=t_n$ vector multiplets,
including the Kaluza--Klein multiplet. With homogeneous
special coordinates $X^\Lambda$, $\Lambda=0,1,\ldots,t_n$,
take the degree-two prepotential
\begin{equation}\label{eq:prepotential}
 F(X)=-\frac{D(X^1,\ldots,X^{t_n})}{X^0}.
\end{equation}
Here $D$ is extended holomorphically as its real-coefficient
cubic polynomial: the squared norms in its real expression
become bilinear quadratic polynomials, with no complex
conjugation in the prepotential.
For $X^0=1$, write $z^I=b^I+i y^I$ with
$y\in\Omega_n^{\rm Cl}$. Direct evaluation gives
\begin{equation}\label{eq:Kahler}
 K=-\log\bigl[i(\bar X^\Lambda F_\Lambda
                  -X^\Lambda\bar F_\Lambda)\bigr]
   =-\log[8D(y)],\qquad
 g_{I\bar J}=\tfrac14\partial_{y^I}\partial_{y^J}[-\log D(y)].
\end{equation}
The prepotential thus gives a positive projective special
K\"ahler metric on this tube \cite{cecotti,dwp,dwvp,cortes}.
The determinant-one linear group, real translations of $b$
and simultaneous dilations act transitively; a dilation changes
$K$ only by a constant. The special K\"ahler structure here
comes from the prepotential. It is distinct from the general
choice of a Bergman metric on a Siegel domain.

For the vector couplings, consider the period matrix.
For $F_{\Lambda\Sigma}=\partial_\Lambda\partial_\Sigma F$,
use the usual period matrix
\begin{equation}\label{eq:periodMatrix}
 \mc N_{\Lambda\Sigma}=\bar F_{\Lambda\Sigma}
 +2i\frac{(\operatorname{Im}F_{\Lambda\Gamma})X^\Gamma
          (\operatorname{Im}F_{\Sigma\Delta})X^\Delta}
 {X^\Theta(\operatorname{Im}F_{\Theta\Pi})X^\Pi}.
\end{equation}
At vanishing axions, $z=iy$, put
$H_{IJ}=-\partial_I\partial_J\log D(y)$. Euler's identities
then give
\begin{equation}\label{eq:periodPositive}
 \mc N_{00}=-iD(y),\quad \mc N_{0I}=0,\quad
 \mc N_{IJ}=-iD(y)H_{IJ}.
\end{equation}
Thus $-\operatorname{Im}\mc N$ is positive definite on
the chosen branch. Real axion translations transport this
positivity by the associated symplectic transformation.
The scalar and vector kinetic terms are therefore positive on
the same branch $y\in\Omega_n^{\rm Cl}$.

There are $t_n+1$ electric field strengths and the same
number of magnetic duals, hence a real symplectic charge
space of dimension $2t_n+2$. The duality group acts symplectically on the equations of motion
and Bianchi identities. A gauging requires a Lie gauge algebra,
an embedding tensor satisfying the closure constraints and a
compatible matter representation. These are additional data to
the ungauged models considered here.

An integral electric--magnetic charge lattice would have rank
$2t_n+2$ and an integral symplectic pairing, whereas
$\Lambda_n$ has rank $N$. Relating the tetracode to charge
quantization or discrete dualities would therefore require a
separate construction compatible with the positive real form
and cubic couplings.

\subsection{Three dimensions and the magical level}
Spacelike reduction and dualization of all vectors implement
the local $c$-map. The resulting quaternionic K\"ahler target
has quaternionic dimension $t_n+1$ and real dimension
$4(t_n+1)$ \cite{cecotti,dwvp}. Timelike reduction gives the corresponding para-quaternionic
K\"ahler geometry.
At $n=1$, the familiar symmetric spaces are, up to finite
central quotients,
\begin{equation}\label{eq:magicSpaces}
 \begin{array}{c|c|c}
 \text{dimension}&\text{scalar space}&\text{scalar dimension}\\\hline
 5&E_{6(-26)}/F_{4(-52)}&26\text{ real}\\
 4&E_{7(-25)}/(E_{6(-78)}\times U(1))&27\text{ complex}\\
 3&E_{8(-24)}/(E_{7(-133)}\times SU(2))&28\text{ quaternionic}
 \end{array}
\end{equation}
For the higher members of $L(8n,1,0)$ the spaces remain
homogeneous and are nonsymmetric. Representative dimensions are
\begin{center}
\begin{tabular}{rrrrrr}\toprule
$n$&$t_n$&$n_V^{(5)}$&$n_V^{(4)}$&$\dim Q$&$\dim_\R\mc M_3$\\\midrule
1&27&26&27&56&112\\
2&275&274&275&552&1104\\
3&4123&4122&4123&8248&16496\\
4&65571&65570&65571&131144&262288\\\bottomrule
\end{tabular}
\end{center}
Here $\dim Q$ is the four-dimensional real charge dimension.
The field counts in this subsequence grow as $2^{4n}=16^n$.
Here $n$ labels the Clifford module and the corresponding model.
The positive Riemannian metrics on these homogeneous scalar
spaces are geodesically complete. This follows from transitivity
of their isometry groups, since connected homogeneous Riemannian
manifolds are complete; for the Clifford hypersurfaces and their
concatenated $q=c\circ r$ map it is also established directly in
\cite{ac2025}. The completeness statement applies to the Riemannian targets
obtained by spacelike reduction.

\subsection{Attractors and the rational dual invariant}
In five dimensions, let $q_I$ denote electric charges and
$Z_5=q_Ih^I$ the central charge in a fixed normalization.
A constrained critical point on $D(h)=1$ obeys
\begin{equation}\label{eq:5dattractor}
 q_I=\lambda\partial_I D(h),\qquad
 \lambda=Z_5/3,
\end{equation}
where Euler's identity $h^I\partial_ID=3D$ gives the second
equation. Let $q^\sharp$ be the vector defined by
$B(q^\sharp,u)=q_Iu^I$. On the positive charge branch
$q^\sharp\in(\Omega_n^{\rm Cl})^*$ with $\lambda>0$,
logarithmic duality \eqref{eq:logDual} gives
\begin{equation}\label{eq:5dinverse}
 \lambda^3=D^*(q^\sharp),\qquad
 h=\lambda\,\nabla_B\log D^*(q^\sharp).
\end{equation}
Indeed $q^\sharp=\lambda\nabla_BD(h)=
\lambda\nabla_B\log D(h)$ on $D=1$, and degree-three
homogeneity plus \eqref{eq:logDual} proves both statements.
At higher levels $D^*$ is rational, so its full derivative enters
this inversion. Other charge branches require their own regularity
and sign analysis.

In four dimensions use charges $Q=(p^\Lambda,q_\Lambda)$
and convention
\begin{equation}\label{eq:centralCharge}
 Z_4=e^{K/2}(q_\Lambda X^\Lambda-p^\Lambda F_\Lambda),
 \qquad V_{\rm BH}=|Z_4|^2+
 g^{I\bar J}D_IZ_4\,\overline{D_JZ_4}.
\end{equation}
Regular supersymmetric attractors satisfy $D_IZ_4=0$;
in units $G_4=1$, their two-derivative entropy is
$S_{\rm BH}=\pi|Z_{4,*}|^2$.
At the magical level the 56-dimensional charge representation
has the exceptional Freudenthal quartic $I_4$, yielding
$S_{\rm BH}=\pi\sqrt{|I_4(Q)|}$ for regular large extremal
orbits in the conventional normalization. The Jordan and Freudenthal identities underlie this quartic
description.

For nonsymmetric homogeneous models, the construction and
inversion of attractor equations are treated in
\cite{ams,amsErr,vanGeemen}. The degree-three dual function
generally has a rational correction, and the 2023 erratum
changes the corresponding derivative and entropy expressions.
For this nonsymmetric family, the charge invariant entering the
entropy calculation need not be a polynomial quartic. The calculation
uses the corrected dual function \eqref{eq:dual} on the selected
physical charge branch. The entropy can involve a square root of
a rational homogeneous charge invariant.

The comparison therefore places the published EP tip product
and its Newton cubic in the known Clifford series, with a
specified positive real form and domain. The identity failures
above level one explain why the magical adjoint and attractor
formulas need modification in the higher models.

\section{Exact checks and reproducibility}\label{sec:checks}
The ancillary directory \texttt{anc/} contains four programs
and their recorded outputs. They check the formulas and counts
at the levels listed below using exact arithmetic. The results
valid at every level have been proved in the preceding sections.

\texttt{exact\_checks.py} uses Python's standard library and
rational arithmetic. It generates roots at $n=1,2,3$,
uses \eqref{eq:coefficients} and \eqref{eq:epsilon}, and checks
the lattice-root census, tetracode ambient and tip fibres,
the fixed frame, idempotents, unit, block dimensions, and
the opposite-pair Jacobiator for a fixed spinorial root
against every other nonopposite spinorial root. It checks
all ordered tip product pairs at $n=1,2$ for closure,
commutativity and degree additivity. The explicit Jacobi
and Jordan witnesses are checked at $n=2,3$.
\begin{center}
\begin{tabular}{rrrrr}\toprule
$n$&$|\Phi|$&$t_n$&tip product pairs&opposite-pair tests\\\midrule
1&240&27&729&126\\
2&2312&275&75625&2046\\
3&33248&4123&0&32766\\\bottomrule
\end{tabular}
\end{center}
The zero in the product column indicates that the exhaustive
pair check was not run at $n=3$. The program also checks the 27 congruence
residues and supplies a scaled $D_8$ frame certificate.

\texttt{enumerate\_gradings.cpp} enumerates all coordinate
value multisets of size $N$ from the nine elements of
$\F_3^2$. Orthogonal root counts are computed directly;
spinorial counts use an exact transfer matrix on degree
and sign parity. The results are
\begin{center}
\begin{tabular}{rrrrr}\toprule
$n$&multisets&balanced nonzero fibres&least kernel roots&
balanced, no spinorial kernel\\\midrule
1&12870&353&18&16\\
2&125970&353&66&16\\
3&735471&978&130&16\\
4&3108105&2274&216&16\\\bottomrule
\end{tabular}
\end{center}
The multiset total is $\binom{N+8}{8}$. ``Kernel roots''
counts only root vectors of degree zero, excluding the
Cartan dimension. The last column at these four levels
is the tetracode family with signed choices on its four
nonzero projective lines. The equivalence relation is coordinate permutation with degree
labels fixed. Classification under the full automorphism group,
or of arbitrary $\F_3^2$ gradings, is a separate problem.

\texttt{clifford\_checks.py} constructs exact real Clifford generators
at $n=1,2$ as signed permutation matrices. It verifies every
generator anticommutator on every basis vector, checks the
Cholesky cubic and inverse diagonal parameters, and tests
both logarithmic duality identities at a specified positive
point. At $n=2$ it also verifies the real non-Jordan witness
and adjoint failure; at $n=1$ the tested adjoint identity holds.
These runs test the two real models at the points just described.
The root-product comparison is handled by the next program.

\texttt{root\_clifford\_bridge\_checks.py} derives its gamma
operators directly from \eqref{eq:tipproduct}.
It checks every vector-generator pair on the entire spinorial
tip basis at $n=1,2,3$, with respectively $1024$, $65536$
and $2359296$ exact anticommutator checks.
It verifies the trace-adjoint relation, the normalization
fixture in Theorem~\ref{th:identification} and the published
tip's signed coordinate map. All $729$ and $75625$
ordered tip products at $n=1,2$ are also checked against
the normalized matrix-block formulas, including the spinor
pairing into $V$. No exhaustive matrix-product pair check
is claimed at $n=3$. The uniform isomorphism and real-form
existence are proved analytically in
Theorem~\ref{th:identification}.

The JSON and text files record the results and their scope.
The Python result files also record the source hashes. Running
the programs requires Python 3 and a C++17 compiler; no
third-party package is needed.

\section{Conclusions and further questions}
The grading and arithmetic constructions separate above level
one. Root-additive brackets continue to respect the tetracode
character, and the $[4,2,3]_3$ code gives balanced fibres at every
level. By contrast, all norm-two vectors then lie in $D_N$.
Orthogonal $A_2$ subsystems need disjoint three-coordinate
supports, which rules out a full-rank unscaled frame. The
Jacobi and Jordan witnesses show directly where the classical
algebraic identities fail.

The associative trace and a root calculation give a Clifford
proof of the published $H$-algebra realization. This provides a
positive real form of the tip and places its Newton cubic in the
known $L(8n,1,0)$ series. Cholesky coordinates identify the positive
cone, and the rational dual invariant gives the corrected
attractor inversion above the magical level.

The next questions concern compatibility and classification.
How do the character gradings split into orbits under the full
automorphism group? Which admit nonuniform refinements or are
compatible with an integral order and the transported real
structure? The spinorial coset also leaves room for integral
and boundary constructions beyond norm-two frames. Any incidence
geometry built from that boundary would need its own points,
lines and incidence theorem.

\section*{Use of generative AI}
During the preparation of this work the authors used Claude (Anthropic)
for code generation, exploratory computations, prose editing and the
preparation of the final source, and ChatGPT (OpenAI) for a mathematical
audit, comparison with the published literature, code generation and rewriting.
The ancillary files document the computational checks and their
scope. The authors are responsible for the mathematical arguments,
the interpretation of the computations and the attribution of results.
\section*{Acknowledgments}
The authors are grateful to P. Truini for introducing D.C.
to Exceptional Periodicity, and for many conversations on Magic Star
algebras and on Vinberg $T$-algebras.


\begin{thebibliography}{99}
\bibitem{adams} J. F. Adams, \emph{Lectures on Exceptional Lie Groups},
University of Chicago Press, 1996.
\bibitem{barton} C. H. Barton and A. Sudbery, Magic squares and matrix models
of Lie algebras, \emph{Adv. Math.} \textbf{180} (2003), 596--647;
\href{https://arxiv.org/abs/math/0203010}{arXiv:math/0203010}.
\bibitem{conway} J. H. Conway and N. J. A. Sloane, \emph{Sphere Packings,
Lattices and Groups}, 3rd ed., Springer, 1999.
\bibitem{corradettiMay} D. Corradetti, Integral elements of Okubo algebra
and the $E_8$-lattice, \href{https://arxiv.org/abs/2605.09333}{arXiv:2605.09333},
version 2 (2026).
\bibitem{corradettiIntegral} D. Corradetti and A. Marrani,
Integral Orders for the Okubo Algebra and Idempotent Geometries of the $E_8$ Lattice,
\href{https://arxiv.org/abs/2608.27762}{arXiv:2608.27762}, version 1 (2026).
\bibitem{coxeter} H. S. M. Coxeter, Integral Cayley numbers,
\emph{Duke Math. J.} \textbf{13} (1946), 561--578.
\bibitem{draperMartin} C. Draper and C. Mart\'in, Gradings on the Albert
algebra and on $\mathfrak f_4$, \emph{Rev. Mat. Iberoam.}
\textbf{25} (2009), 841--908.
\bibitem{elduqueComposition} A. Elduque, Composition algebras,
in A. Makhlouf (ed.), \emph{Algebra and Applications 1}, ISTE--Wiley,
2021, pp. 27--57; \href{https://arxiv.org/abs/1810.09979}{arXiv:1810.09979}.
\bibitem{elduqueOrder} A. Elduque, Order three elements in $G_2$ and
idempotents in symmetric composition algebras, \emph{Canad. J. Math.}
\textbf{70} (2018), 1038--1075;
\href{https://arxiv.org/abs/1705.01293}{arXiv:1705.01293}.
\bibitem{elduqueSymmetric} A. Elduque, Gradings on symmetric composition
algebras, \emph{J. Algebra} \textbf{322} (2009), 3542--3579;
\href{https://arxiv.org/abs/0809.1922}{arXiv:0809.1922}.
\bibitem{elduqueOkubo} A. Elduque, Okubo algebras: automorphisms,
derivations and idempotents, \emph{Contemp. Math.} \textbf{652} (2015),
61--73; \href{https://arxiv.org/abs/1310.5455}{arXiv:1310.5455}.
\bibitem{elduqueJordan} A. Elduque, Jordan gradings on exceptional simple
Lie algebras, \emph{Proc. Amer. Math. Soc.} \textbf{137} (2009),
4007--4017; \href{https://arxiv.org/abs/0810.2737}{arXiv:0810.2737}.
\bibitem{elduqueKochetov} A. Elduque and M. Kochetov,
\emph{Gradings on Simple Lie Algebras}, Mathematical Surveys and
Monographs 189, American Mathematical Society, 2013.
\bibitem{elduqueMyung} A. Elduque and H. C. Myung, On flexible composition
algebras, \emph{Comm. Algebra} \textbf{21} (1993), 2481--2505.
\bibitem{faraut} J. Faraut and A. Kor\'anyi, \emph{Analysis on Symmetric Cones},
Oxford University Press, 1994.
\bibitem{jvnw} P. Jordan, J. von Neumann and E. Wigner, On an algebraic
generalization of the quantum mechanical formalism,
\emph{Ann. of Math.} \textbf{35} (1934), 29--64.
\bibitem{koecher} M. Koecher, Positivit\"atsbereiche im $\R^n$,
\emph{Amer. J. Math.} \textbf{79} (1957), 575--596.
\bibitem{kollross} A. Kollross, The bracket of the exceptional Lie
algebra $E_8$, \href{https://arxiv.org/abs/2504.16513}{arXiv:2504.16513},
version 3 (16 May 2026).
\bibitem{loos} O. Loos, \emph{Jordan Pairs}, Lecture Notes in Mathematics
460, Springer, 1975.
\bibitem{marraniT} A. Marrani, Vinberg's $T$-algebras. From Exceptional
Periodicity to black hole entropy, \emph{SciPost Phys. Proc.}
\textbf{14} (2023), 035;
\href{https://arxiv.org/abs/2312.12371}{arXiv:2312.12371}.
\bibitem{marraniTruini} A. Marrani and P. Truini, Exceptional Lie algebras,
$SU(3)$ and Jordan pairs. Part 2: Zorn-type representations,
\emph{J. Phys. A} \textbf{47} (2014), 265202;
\href{https://arxiv.org/abs/1403.5120}{arXiv:1403.5120}.
\bibitem{okubo} S. Okubo, Pseudo-quaternion and pseudo-octonion algebras,
\emph{Hadronic J.} \textbf{1} (1978), 1250--1278.
\bibitem{petersson} H. P. Petersson, Eine Identit\"at f\"unften Grades,
der gewisse Isotope von Kompositions-Algebren gen\"ugen,
\emph{Math. Z.} \textbf{109} (1969), 217--238.
\bibitem{springer} T. A. Springer and F. D. Veldkamp,
\emph{Octonions, Jordan Algebras and Exceptional Groups}, Springer, 2000.
\bibitem{truini} P. Truini, Exceptional Lie algebras, $SU(3)$ and Jordan
pairs, \emph{Pacific J. Math.} \textbf{260} (2012), 227--243;
\href{https://arxiv.org/abs/1112.1258}{arXiv:1112.1258}.
\bibitem{ep1} P. Truini, A. Marrani and M. Rios,
Exceptional Periodicity and Magic Star algebras. I: Foundations,
\href{https://arxiv.org/abs/1909.00357}{arXiv:1909.00357}.
\bibitem{ep2} P. Truini, A. Marrani and M. Rios,
Exceptional Periodicity and Magic Star algebras. II: Gradings and
HT-algebras, \href{https://arxiv.org/abs/1910.07914}{arXiv:1910.07914}.
\bibitem{epPublished} P. Truini, A. Marrani, M. Rios and W. de Graaf,
Exceptional Periodicity and Magic Star algebras: generalized roots,
gradings, Hermitian Vinberg $T$-algebras and their derivations,
\emph{Expo. Math.} \textbf{43} (2) (2025), 125621;
\href{https://doi.org/10.1016/j.exmath.2024.125621}{doi:10.1016/j.exmath.2024.125621}.
\bibitem{epQuantum} P. Truini, A. Marrani and M. Rios, Magic Star and
Exceptional Periodicity: an approach to quantum gravity,
\emph{J. Phys. Conf. Ser.} \textbf{1194} (2019), 012106;
\href{https://arxiv.org/abs/1811.11202}{arXiv:1811.11202}.
\bibitem{epStar} P. Truini, M. Rios and A. Marrani, The Magic Star of
Exceptional Periodicity, \emph{Contemp. Math.} \textbf{721} (2019),
277--297; \href{https://arxiv.org/abs/1711.07881}{arXiv:1711.07881}.
\bibitem{vinberg} E. B. Vinberg, The theory of convex homogeneous cones,
\emph{Trans. Moscow Math. Soc.} \textbf{12} (1963), 340--403.
\bibitem{ac} D. V. Alekseevsky and V. Cort\'es, Special Vinberg cones,
\emph{Transform. Groups} \textbf{26} (2021), 377--402;
\href{https://doi.org/10.1007/s00031-021-09649-w}{doi:10.1007/s00031-021-09649-w}.
\bibitem{ams} D. V. Alekseevsky, A. Marrani and A. Spiro, Special
Vinberg cones and the entropy of BPS extremal black holes,
\emph{JHEP} \textbf{11} (2021), 100;
\href{https://arxiv.org/abs/2107.06797}{arXiv:2107.06797}, corrected version 4.
\bibitem{ac2025} D. V. Alekseevsky and V. Cort\'es, Homogeneous
cones associated with Clifford modules, in D. R. Wood, A. M. Etheridge,
J. de Gier and N. Joshi (eds.), \emph{2023 MATRIX Annals}, MATRIX
Book Series 6, Springer, 2025, pp. 113--124;
\href{https://doi.org/10.1007/978-3-031-76738-8_7}{doi:10.1007/978-3-031-76738-8\_7}.
\bibitem{rank4} D. V. Alekseevsky and P. Osipov,
Special Vinberg cones of rank 4, \emph{JHEP} \textbf{01} (2026), 001;
\href{https://doi.org/10.1007/JHEP01(2026)001}{doi:10.1007/JHEP01(2026)001}.
\bibitem{amsErr} D. V. Alekseevsky, A. Marrani and A. Spiro,
Erratum: Special Vinberg cones and the entropy of BPS extremal
black holes, \emph{JHEP} \textbf{06} (2023), 192;
\href{https://doi.org/10.1007/JHEP06(2023)192}{doi:10.1007/JHEP06(2023)192}.
\bibitem{cubics} D. V. Alekseevsky, A. Marrani and A. Spiro,
Special Vinberg cones, invariant admissible cubics and special
real manifolds, in \emph{The Diverse World of PDEs: Geometry and
Mathematical Physics}, \emph{Contemp. Math.} \textbf{788} (2023), 31--51;
\href{https://doi.org/10.1090/conm/788/15818}{doi:10.1090/conm/788/15818};
\href{https://arxiv.org/abs/2301.01168}{arXiv:2301.01168}, version 2.
\bibitem{abs} M. F. Atiyah, R. Bott and A. Shapiro,
Clifford modules, \emph{Topology} \textbf{3}, suppl. 1 (1964), 3--38.
\bibitem{gst} M. G\"unaydin, G. Sierra and P. K. Townsend,
The geometry of $N=2$ Maxwell--Einstein supergravity and
Jordan algebras, \emph{Nucl. Phys. B} \textbf{242} (1984), 244--268.
\bibitem{cecotti} S. Cecotti, Homogeneous K\"ahler manifolds and
$T$-algebras in $N=2$ supergravity and superstrings,
\emph{Comm. Math. Phys.} \textbf{124} (1989), 23--55.
\bibitem{dwp} B. de Wit and A. Van Proeyen, Special geometry,
cubic polynomials and homogeneous quaternionic spaces,
\emph{Comm. Math. Phys.} \textbf{149} (1992), 307--334;
\href{https://arxiv.org/abs/hep-th/9112027}{arXiv:hep-th/9112027}.
\bibitem{dwvp} B. de Wit, F. Vanderseypen and A. Van Proeyen,
Symmetry structure of special geometries,
\emph{Nucl. Phys. B} \textbf{400} (1993), 463--524;
\href{https://arxiv.org/abs/hep-th/9210068}{arXiv:hep-th/9210068}.
\bibitem{cortes} V. Cort\'es, Homogeneous special geometry,
\emph{Transform. Groups} \textbf{1} (1996), 337--373.
\bibitem{vanGeemen} B. van Geemen, A. Marrani and F. Russo,
BPS black hole entropy and attractors in very special geometry.
Cubic forms, gradient maps and their inversion,
\emph{JHEP} \textbf{12} (2021), 195;
\href{https://arxiv.org/abs/2009.10647}{arXiv:2009.10647}.
\end{thebibliography}
\end{document}